\documentclass{amsart}

\usepackage{amsmath}
\usepackage{amssymb}
\usepackage{amsfonts}
\usepackage{amsopn}
\usepackage{amsthm}
\usepackage[all]{xy} 
\usepackage{hyperref}

\swapnumbers
\theoremstyle{plain}
\newtheorem{thm}{Theorem}[section]
\newtheorem{lem}[thm]{Lemma}

\theoremstyle{definition}

\newtheorem{ex}[thm]{Example}

\DeclareMathOperator{\CH}{CH}  
\DeclareMathOperator{\K0}{K_0}   
\DeclareMathOperator{\ind}{ind} 

\newcommand{\XX}{{}_\xi X}  

\newcommand{\ZZ}{\mathbb{Z}}   

\newcommand{\cc}{\mathfrak{c}}  

\newcommand{\LL}{\mathcal{L}}   

\newcommand{\ia}{{\mathrm{i}_a}}  

\newcommand{\ib}{{\mathrm{i}_b}}  

\newcommand{\ic}{{\mathrm{i}_c}}  

\newcommand{\istar}{{\mathrm{i}_*}} 

\DeclareMathOperator{\res}{res}   

\DeclareMathOperator{\igam}{\gamma^{{\it i/i}+1}K_0(X)}  

\DeclareMathOperator{\igamG}{\gamma^{\it i}\mathfrak{G}_s}  

\DeclareMathOperator{\igtwist}{\gamma^{\it i}_\xi\mathfrak{G}_s}  

\title{Twisted gamma filtration and algebras with orthogonal involution}

\author{Caroline Junkins}

\begin{document}


\begin{abstract}
For the Grothendieck group of a split simple linear algebraic group, the
twisted $\gamma$-filtration provides a useful tool for constructing torsion
elements in $\gamma$-rings of twisted flag varieties. In this paper, we construct a non-trivial torsion element in the $\gamma$-ring of a complete flag variety twisted
by means of a $PGO$-torsor. This generalizes the construction in the
$HSpin$ case previously obtained by Zainoulline.
We use this torsion element to study algebras with orthogonal
involutions.
\end{abstract}

\maketitle

\section*{Introduction}

Let $X$ be a variety of Borel subgroups of a simple linear algebraic group $G$ over a field $k$. Consider the Grothendieck group $\K0(X)$ and the associated graded ring $\gamma^*\K0(X)$ whose degree $i$ components are determined by the Grothendieck $\gamma$-filtration. The ring $\gamma^*\K0(X)$ was introduced by Grothendieck as an approximation of the topological filtration on $\K0$, which in turn is related to the Chow groups $\CH^*(X)$ of algebraic cycles modulo rational equivalence. By the Riemann-Roch theoreom (cf. \cite{SGA6}), there exists a surjection $\gamma^{2/3}\K0(X)\twoheadrightarrow \CH^2(X)$ (cf. \cite[Ex. 15.3.6]{Fu}), and so the torsion in $\gamma^2\K0$ may be viewed as an upper bound for the torsion in $\CH^2$. 

Determining torsion in $\CH^d(X)$ is a non-trivial problem, and only partial results are known. For $d=2,3$ and $G$ stronly inner, we refer to \cite{Pe98} and \cite{GaZ10}. The case of quadrics was considered in \cite{KM} for $d=2,3,4$. The $\gamma$-filtration was used in \cite{Ka98}, and in \cite{Ka96} it was found that $Tors\CH^4$ can in fact be infinitely generated. Results for arbitrary $d$ have been obtained recently in \cite{BZZ} and \cite{Ba12}, by providing upper bounds for the annihilators of $Tors\CH^d$.

The $\gamma$-filtration has also proved useful for studying the motivic J-invariant in \cite{QSZ} and \cite{J11}.
In the present paper we use the twisted $\gamma$-filtration, introduced in \cite{Z12}, to provide a non-trivial torsion element in $\gamma^2\K0(X)$, where $X$ is a complete flag variety twisted by means of a $PGO$-torsor. Namely, we prove the following

\begin{thm}
Let $X$ be the variety of Borel subgroups for a linear algebraic group $G$ of type $D_n$, for $n\geq 4$ even. Suppose that the indicies $2^\ia\leq 2^\ib\leq  2^\ic$ of the Tits algebras of $G$ are all non-trivial, are bounded above by $n$, and that they satisfy the additional condition that $\ia+\ib>\ic\geq 2$. Under these conditions, there exists a non-trivial torsion element in $\gamma^2\K0(X)$ which is nontrivial in the twisted $\gamma$-filtration of $\K0(X)$ and which vanishes over a splitting field of $G$. 
\end{thm}

Furthermore, we may relate this element to a torsion element provided in \cite{Z12} for orthogonal groups, by considering a finite field extension of $k$ over which one of the Tits algebras of $G$ splits. 

This paper is organized as follows. In the first section we recall the definition of the $\gamma$-filtration and provide a description through use of the Steinberg basis and Tits algebras of $G$. In the second and third sections we introduce the twisted $\gamma$-filtration $\gamma^*_\xi\mathfrak{G}_s$, and compute $\gamma^{2/3}_\xi\mathfrak{G}_s$ in the case that $G=PGO_{2n}$ for $n\geq 4$ even. In the final section, we construct an explicit torsion element in $\gamma^2\K0(X)$, provide an example where such an element exists, and finally consider the behaviour of such an element over a finite field extension.

\section{Preliminaries}

\noindent In this section we recall the constructions and results
which will be neccessary for the main theorem. We recall the
definition of Grothendieck's $\gamma$-filtration, introduce the notion
of the Steinberg basis for $K_0$ of a complete flag variety and recall
basic properties of $K_0$ of twisted flag varieties. The reader is advised to consult \cite[Chap. 15]{Fu}, \cite[Chap. 3]{FL} or \cite{GaZ10} for further details.

Consider a smooth projective variety $X$ over a field $k$
and its Grothendieck group $\K0(X)$. For an element $x\in\K0(X)$, let $c(x)=\sum_{i\ge0}c_i(x)$ be the total Chern class of $x$ with values in $\K0(X)$ (cf. \cite{Fu}). We follow the convention that for a line bundle $L$ over $X$,  $c_1([L])=1-[L^\vee]$.

The tensor product of vector bundles endows $\K0(X)$ with a ring structure, and while it does not have a canonical grading, we may impose one through the construction of a filtration. The $\gamma$-filtration is defined using the above Chern classes. We define
$$
\gamma^i\K0(X)=
\langle c_{i_1}(x_1)\cdot \ldots \cdot c_{i_m}(x_m) \mid
i_1+\ldots + i_m\ge i,\; x_l\in \K0(X)\rangle,
$$
and let $\igam=\gamma^i\K0(X)/\gamma^{i+1}\K0(X)$ denote the degree
$i$ component  of the corresponding graded ring (cf. \cite[\S 2.3]{SGA6}).

Consider a split simple linear algebraic group $G_s$ of rank $n$ over
$k$. We fix a split maximal torus $T$ and Borel subgroup $B$ such that
$T\subset B\subset G_s$. Let $X_s=G_s/B$ be the variety of Borel
subgroups. Let $\K0(X_s)$ be the Grothendieck group of $X_s$ and
recall that since $G_s$ is split, $\K0(X_s)$ is generated by line
bundles (see \cite{SGA6}).

Let $\Lambda_r$ and $\Lambda$ be the root and weight lattices of $G_s$ respectively. Let $\{\alpha_1,\dots, \alpha_n\}$ be a set of simple roots and $\{\omega_1, \dots, \omega_n\}$ the respective set of fundamental weights, that is $\alpha_i^\vee(\omega_j)=\delta_{ij}$. The character group $T^*$ of $T$ is an intermediate lattice $\Lambda_r\subseteq T^*\subseteq \Lambda$ which determines the isogeny class of $G_s$. For instance if $T^*=\Lambda_r$ then $G_s$ is adjoint, and if $T^*=\Lambda$ then $G_s$ is simply connected. 

With this, we may define the integral group ring $\ZZ[T^*]$ whose elements are linear combinations $\sum_ia_ie^{\lambda_i}$, for $a_i\in\ZZ$ and $\lambda_i\in T^*$. Let $$
\cc: \ZZ[T^*]\rightarrow \K0(X_s)
$$ 
be the characteristic map, defined by sending $e^\lambda$ to the
associated line bundle $[\mathcal{L}(\lambda)]$ for $\lambda\in
\Lambda$. Note that while $\K0(X_s)$ does not depend on the isogeny
class of $G_s$, the image of this map does. In particular, if $G_s$ is
simply connected then the characteristic map is surjective \cite{SGA6}. 

The Weyl group $W$ of $G_s$ acts on weights via simple reflections $s_{\alpha_i}$, $i=1,\dots,n$. That is $s_{\alpha_i}(\lambda)=\lambda-\alpha_i^\vee(\lambda)\alpha_i$, for $\lambda\in\Lambda$. So, for each element $w\in W$ we may define the weight $\rho_w\in\Lambda$ (cf. \cite[\textsection 2.1]{St75}) as
$$
\rho_w=\sum_{\{i\in 1,\dots,n\mid w^{-1}(\alpha_i)<0\}} w^{-1}(\omega_i).
$$
In particular if $w=s_{\alpha_i}$, then $\rho_w=s_{\alpha_i}(\omega_i)=\omega_i-\alpha_i$. Since $W$ acts trivially on $\Lambda/\Lambda_r$, we have 
$$
\overline{\rho}_w=\sum_{\{i\in 1,\dots,n\mid w^{-1}(\alpha_i)<0\}}\overline{\omega}_i \in \Lambda/T^*,
$$
where $\overline{\rho}_w$ denotes the class of $\rho_w\in \Lambda$ modulo $T^*$.
By the characteristic map of the simply connected cover of $G_s$, we may associate to each $w\in W$ the class of the associated line bundle $g_w:=\cc(e^{\rho_w})=[\LL(\rho_w)]$. These elements form a $\ZZ$-basis of $\K0(X_s)$ called the Steinberg basis.
The Steinberg basis allows a nice description of the $\gamma$-filtration of $X_s$, since we may express generators in terms of only first Chern classes:
$$
\gamma^i\K0(X_s)=\langle c_1(g_{w_1})\cdot \ldots\cdot c_1(g_{w_m}) \mid m\ge i, w_1,\dots,w_m\in W\rangle.
$$
Unfortunately, for a cocyle $\xi\in Z^1(k, G_s)$ the twisted flag variety $X=\XX_s$ is not in general generated by line bundles, and so a description of $\gamma^i\K0(X)$ is not as straightforward.

For a fixed $\xi\in Z^1(k,G_s)$, we can associate to each weight $\overline{\lambda}\in\Lambda/T^*$ a central simple $k$-algebra $A_{\xi,\overline{\lambda}}$, called a Tits algebra of $G_s$ with index $\ind(A_{\xi,\overline{\lambda}})$ (cf. \cite{Ti71}). We define the Tits map 
$$
\beta_\xi:\Lambda/T^*\rightarrow Br(k)
$$
by sending $\overline{\lambda}$ to the class
$[A_{\xi,\overline{\lambda}}]$ in the Brauer group. This map is a group homomorphism for a fixed $\xi$, with $\overline{\lambda_1}+\overline{\lambda_2}\mapsto [A_{\xi,\overline{\lambda_1}}]\otimes[A_{\xi,\overline{\lambda_2}}]$.

Let $X={}_\xi X_s$, let $k_{sep}$ be the separable closure of $X$, and consider the restriction map
$$
\res\colon \K0(X)\to \K0(X\times_k k_{sep})=\K0(X_s),
$$
where we identify $\K0(X\times_k k_{sep})$ with $\K0(X_s)$. 
The main result of \cite{Pa94} says that the image of this restriction map coincides with the sublattice $\langle \ind(A_w) g_w \rangle_{w\in W},$ where $g_w$ is an element of the Steinberg basis and $\ind(A_w)\ge 1$ is the index of the respective Tits algebra $A_{\xi,\overline{\rho_w}}$. Note that if $G_s$ is simply connected the restriction map is an isomorphism.

Since characteristic classes commute with restrictions, the
restriction map $\res\colon \gamma^i\K0(X)\rightarrow \gamma^i\K0(X_s)$ is well defined, and
$$
\res(\gamma^i\K0(X))=\left\langle\prod_{j=1}^m\binom{\ind(A_{w_j})}{n_j}c_1(g_{w_j})^{n_j}\mid n_1+\dots +n_m\ge i, w_j\in W\right\rangle.
$$


\section{The twisted \texorpdfstring{$\boldsymbol{\gamma}$}{gamma}-filtration}

While we have an explicit description of the image $\res(\gamma^{2/3}\K0(X))$ in terms of the Steinberg basis and the indices of the Tits algebras, it is not as practical to work with as one would hope. One problem is that expressing the tensor product of two Steinberg elements as a linear combination of Steinberg elements is a non-trivial task (especially when $W$ is large). 
The twisted $\gamma$-filtration was introduced in \cite{Z12} as a tool for getting identifying elements in $\gamma^{2/3}\K0(X)$ more easily. 

Returning to the characteristic map and the canonical surjection $\Lambda\mapsto \Lambda/T^*$, we have the following diagram, 
$$
\xymatrix{
\ZZ[\Lambda] \ar@{->>}[r]^{\cc} \ar@{->>}[d]	& \K0(X_s) \ar[rd]^q \ar[r]^\simeq	&\ZZ[\Lambda]/\ker(\cc) \ar@{->>}[d]\\
\ZZ[\Lambda/T^*] \ar@{->>}[rr]		&				& \ZZ[\Lambda/T^*]/\overline{\ker(\cc)}
}
$$

which allows us to define the quotient ring
$$
\mathfrak{G}_s:=\ZZ[\Lambda/T^*]/\overline{\ker(\cc)},
$$
and the composite map $q:\K0(X_s)\rightarrow \mathfrak{G}_s$, which is a surjective ring homomorphism. Observe that if $G_s$ is simply connected then $\mathfrak{G}_s\simeq\ZZ$.

By Lemma 3.3 in \cite{Z12} the ideal $\overline{\ker(\cc)}\subset \ZZ[\Lambda/T^*]$ is generated by the elements $d_i(1-e^{\bar{\omega}_i})$, $i=1,\dots,n$, where $d_i$ is the number of elements in the $W$-orbit of the fundamental weight $\omega_i$. 

Consider next the $\gamma$-adic filtration on $\ZZ[\Lambda]$, where the $i$-th power $\gamma^i$ is generated by products of at least $i$ differences. That is, 
$$
\gamma^i=\langle (1-e^{\lambda_1})\dots(1-e^{\lambda_k})\mid k\ge i\rangle.
$$
With this we have $\gamma^i\K0(X_s)=\cc(\gamma^i)$ and we define $\igamG:=q(\gamma^i\K0(X_s)), i\ge 0$.

Given a fixed $G_s$-torsor $\xi$ and the respective twisted form $X={}_\xi X_s$, define the twisted $\gamma$-filtration on $\mathfrak{G}_s$ to be:
$$
\igtwist:=q(res(\gamma^i\K0(X))), i\ge 0,
$$
and let $\gamma^{\it{i/i+1}}_\xi\mathfrak{G}_s=\gamma^{\it i}_\xi\mathfrak{G}_s/\gamma^{\it i+1}_\xi\mathfrak{G}_s$. The associated graded ring $\gamma_\xi^*:=\bigoplus_{i\ge 0}\gamma_\xi^{i/i+1}\mathfrak{G}_s$ is called the $\gamma$-invariant of $\xi$.

Futhermore, there exists a surjective ring homomorphism $\gamma^*(X)\twoheadrightarrow \gamma_\xi^*$, so we may provide a description of the twisted $\gamma$-filtration similar to that in the previous section. By Theorem 4.5 in \cite{Z12},
$$
\igtwist=\left\langle \prod_{j=1}^m \binom{\ind(A_{w_j})}{n_j}(1-e^{-\bar{\rho}_{w_j}})^{n_j}\mid n_1+\dots+n_m\ge i, w_j\in W\right\rangle
$$

\section{Algebras with orthogonal involution}

Let $G_s$ be the adjoint group $PGO^+_{2n}$, for $n\geq 4$ even; ie. $G$ is of type $D_{n}$, $n$ even. In this case,
$\Lambda/T^*=\Lambda/\Lambda_r=\{0, \overline{\omega}_1,
\overline{\omega}_{n-1}, \overline{\omega}_n\}\simeq \ZZ/2\ZZ\oplus
\ZZ/2\ZZ$. That is, $2\bar{\omega}_1=0$,
$\bar{\omega}_1=\bar{\omega}_{n-1}+\bar{\omega}_n$, and
$\bar{\omega}_s=s\bar{\omega}_1$ for $2\leq s\leq n-2$. Let $C_+$ and
$C_-$ denote the Tits algebras corresponding to
$\overline{\omega}_{n-1}$ and $\overline{\omega}_n$. Let $A$ denote
the Tits algebra corresponding to $\overline{\omega}_1$. We note that
$C_+\times C_-$ is the even part of the Clifford algebra of the
algebra with involution $A$, and $[A]+[C_+]+[C_-]=0$ in $Br(k)$ (cf.\cite{INV}).

In order to simplify some calculations, we choose the generators $\sigma_1$ and $\sigma_2$ from $\{\overline{\omega}_1, \overline{\omega}_{n-1}, \overline{\omega}_n\}$ such that $\ind(\beta_\xi(\sigma_1))\leq \ind(\beta_\xi(\sigma_2))\leq \ind(\beta_\xi(\sigma_1+\sigma_2))$. For ease of notation we will define the following non-negative integers:
$$
\ia=v_2(\ind(\beta_\xi(\sigma_1))),\;\;\; \ib=v_2(\ind(\beta_\xi(\sigma_2))),\;\;\; \ic=v_2(\ind(\beta_\xi(\sigma_1+\sigma_2)))
$$
Thus $\ia\leq\ib\leq\ic$, and by the relation in the Brauer group, we have $\ic\leq \ia+\ib$. Then letting $y_1=1-e^{\sigma_1}$ and $y_2=1-e^{\sigma_2}$, we have
$$
\mathfrak{G}_s\cong \ZZ[y_1,y_2]/(y_1^2-2y_1, y_2^2-2y_2, d_a y_1, d_b y_2, d_c(y_1+y_2-y_1y_2)),
$$
where $d_a, d_b, d_c$ are determined by the number of elements in the $W$-orbit of the three corresponding fundamental weights. 

\begin{ex}
If $n=4$ (ie. $G_s=PGO^+_8$), then $d_a=d_b=d_c=2^3$, and so
$$
\mathfrak{G}_s\cong \ZZ[y_1,y_2]/(y_1^2-2y_1, y_2^2-2y_2, 8y_1, 8y_2).
$$
\end{ex}

To compute $\gamma^{2/3}_\xi\mathfrak{G}_s$ we first make note of 4 cases, dependent on the indices of the Tits algebras. The arguments used are independent of the choice of generator, so without loss of generality we consider the generator $y_1$ and its associated index $2^\ia$. 

Consider first the case that the associated Tits algebra is split, ie. $\ia=0$. Then, $y_1^2\equiv 2y_1\in\gamma^2_\xi\mathfrak{G}_s$, and $y_1^3\equiv 4y_1\in\gamma^3_\xi\mathfrak{G}_s$. Thus $\gamma^{2/3}_\xi\mathfrak{G}_s$ has a summand isomorphic to $\ZZ/2\ZZ$. 

Next, suppose $\ia=1$. Then $\binom{2}{2}y_1^2\equiv 2y_1\in\gamma^2_\xi\mathfrak{G}_s$, and $\binom{2}{2}\binom{2}{1}y_1^3\equiv 8y_1\in\gamma^3_\xi\mathfrak{G}_s$. This time, $\gamma^{2/3}_\xi\mathfrak{G}_s$ has a summand isomorphic to $\ZZ/4\ZZ$.

Thirdly, consider the case that $2\leq\ia<v_2(d_a)$.  Then $\binom{2^\ia}{2}y_1^2\equiv 2^\ia y_1\in\gamma^2_\xi\mathfrak(G)_s$, and $\binom{2^\ia}{4}y_1^4\equiv 2^{\ia+1}y_1\in\gamma^4_\xi\mathfrak{G}_s\subseteq \gamma^3_\xi\mathfrak{G}_s$. So, $\gamma^{2/3}_\xi\mathfrak{G}_s$ has a summand isomorphic to $\ZZ/2\ZZ$

Finally, suppose $\ia\ge v_2(d_a))$. In this situation we have $\binom{2^\ia}{r}y_1^r\mid 2^\ia y_1\equiv 0$ for all $r\geq 2$, and so the summand is trivial.

For $y_1+y_2-y_1y_2$, the results are slightly different for $\ic\geq 2$. We have $2^\ic (y_1+y_2-y_1y_2)\in\gamma^2\mathfrak{G}_s$ as before; however, in $\gamma^4\mathfrak{G}_s$ we have the additional element ${2^\ia}{2}\binom{2^\ib}{2}y_1^2y_2^2\equiv 2^{\ia+\ib}y_1y_2$. By the relations in the Brauer group, $\ia+\ib\ge \ic$, and so we have a non-trivial multiple of $y_1+y_2-y_1y_2$ only if this inequality is strict, or if $\ia=0$. In the specific case that $\ia=\ib=1, \ic=2$, we have $y_1^2y_2^2\equiv 4y_1y_2\in\gamma^3\mathfrak{G}_s$, but $4y_1, 4y_2\notin\gamma^3\mathfrak{G}_s$. 

We may summarize this as follows. 

\begin{lem}
Let $\gamma^{2/3}_\xi\mathfrak{G}_s=G_a\oplus G_b\oplus G_c$. Then we have
\begin{align*}
G_a\simeq \begin{cases}
	\ZZ/2\ZZ	&\text{if $\ia=0$}\\
	\ZZ/4\ZZ	&\text{if $\ia=1$}\\
	\ZZ/2\ZZ	&\text{if $2\leq \ia<v_2(d_a)$}\\
	0		&\text{if $\ia\geq v_2(d_a)$}
		\end{cases}
\;\;\;\;\;\;\;G_b\simeq \begin{cases}
	\ZZ/2\ZZ	&\text{if $\ib=0$}\\
	\ZZ/4\ZZ	&\text{if $\ib=1$}\\
	\ZZ/2\ZZ	&\text{if $2\leq \ib<v_2(d_b)$}\\
	0		&\text{if $\ib\geq v_2(d_b)$}
		\end{cases}\\
\;\;\;G_c\simeq \begin{cases}
	\ZZ/2\ZZ	&\text{if $\ic=0$}\\
	\ZZ/4\ZZ	&\text{if $\ic=1$}\\
	\ZZ/2\ZZ &\text{if $\ic=2, \ia=\ib=1$}\\
	\ZZ/2\ZZ	&\text{if $2\leq\ic<v_2(d_c), \ia+\ib<\ic$}\\
	0	&\text{if $2\leq \ic<v_2(d_c), \ic=\ia+\ib$}\\
	0		&\text{if $\ic\geq v_2(d_c)$}
		\end{cases}\\
\end{align*}
\end{lem}

\begin{ex}
Let $G_s=HSpin_{2n}$ be a half-spin group of rank $n\ge 4$. Then,
$G_s$ is of type $D_n$ for $n$ even, and we have $\Lambda/T^*=\langle
\sigma\rangle$ where $\sigma$ is of order 2 \cite{INV}.

This corresponds to taking the quotient of $\Lambda/\Lambda_r\simeq\langle \sigma_1\rangle\oplus\langle\sigma_2\rangle$ modulo one of the generators, eg. $\sigma_2\equiv 0$. In this situation, we have
$$
\mathfrak{G}_s\simeq \ZZ[y_1]/(y_1^2-2y_1, dy_1),
$$
where $d=2^{v_2(n)+1}$. Thus we may describe the twisted $\gamma$-filtration of $G_s$ by taking the quotient of our previous description of $\gamma^{2/3}\mathfrak{G}_s$ for $G_s=PGO_{2n}^+$ by  $\langle 2^{v_2(n)+1}y_1, y_2\rangle$. So, 
$$
\gamma^{2/3}\mathfrak{G}_s\simeq\begin{cases}
	\ZZ/2\ZZ  &\text{ if $\ia=0$}\\
	\ZZ/4\ZZ  &\text{ if $\ia=1$}\\
	\ZZ/2\ZZ &\text{ if $2\leq \ia\leq v_2(n)$}\\
	0		&\text{ if $]ia> v_2(n)$}
	\end{cases}
$$
Note that this corresponds to the result given in \cite[Example 4.8]{Z12}.
\end{ex}


\section{Construction of a torsion element}

In this section, we provide an explicit torsion element in
$\gamma^{2/3}\K0(X)$ using the twisted $\gamma$-filtration.
Namely, we prove the following

\begin{thm}
Assume that $0<\ia\leq \ib\leq \ic\leq v_2(n)$, and that $\ia+\ib>\ic\geq 2$. Then there exists a non-trivial torsion element of order dividing $2^{\ia+\ib-\ic}$ in $\gamma^{2/3}\K0(X)$. Furthermore, its image in $\gamma_\xi^{2/3}\mathfrak{G}_s$ via $q$ is non-trivial, and its image in $\gamma^{2/3}\K0(X_s)$ via $\res$ is trivial.
\end{thm}
\begin{proof}
By \cite[Cor. 5.5]{GZ12}, $\cc(\ZZ[T^*])\subset \K0(X)$, so we begin by considering classes of bundles lying in this image.

Denote by $g_a, g_b, g_c$ classes of line bundles $\LL(\lambda)$, where $\bar{\lambda}=\sigma_1, \sigma_2$ and $\sigma_1+\sigma_2$ respectively. Then, $2\lambda\in T^*$ in each case and so $c_1(g^2)=c_1([\LL(2\lambda)])\in\gamma^1\K0(X)$ for $g=g_a, g_b$ and $g_c$. In particular, the following element then lies in $\gamma^2\K0(X)$:
\begin{align*}
c_1(g_a^2)c_1(g_b^2)=4c_1(g_a)c_1(g_b)-2c_1(g_a)c_1(g_b)^2-2c_1(g_a)^2c_1(g_b)+c_1(g_a)^2c_1(g_b)^2\\
=c_1(g_a)^2c_1(g_b)^2+2c_1(g_a)c_1(g_b)c_1(g_c)-2c_1(g_a)^2-2c_1(g_b)^2+2c_1(g_c)^2.
\end{align*}
By the retriction map, $c_2(i_w g_w)=\binom{i_w}{2}c_1(g_w)^2\in\gamma^2\K0(X)$ for all $w\in W$. So, under the hypothesis that $\ic\ge 2$, we may consider the following element in $\gamma^2\K0(X)$:
$$
\mu=2^{\ic-2}c_1(g_a^2)c_1(g_b^2)+c_2(2^\ic g_a)+c_2(2^\ic g_b)-c_2(2^\ic g_c)\in \gamma^2\K0(X).
$$
Note that we may permute the line bundles $g_a, g_b, g_c$ in the definition of $\mu$ (but keep the coefficients as $2^{\ic-2}$ and $2^\ic$ respectively), and we will still have an element of $\gamma^2\K0(X)$. In fact, all subsequent arguments are valid for such permutations as well.

By Chern class operations and the relation $c_1(g_c)=c_1(g_a)+c_1(g_b)-c_1(g_a)c_1(g_b)$,
\begin{align*}
\mu&=2^{\ic-2}c_1(g_a^2)c_1(g_b^2)+2^{\ic-1}c_1(g_a)^2+2^{\ic-1}c_1(g_b)^2-2^{\ic-1}c_1(g_c)^2\\
&=2^{\ic-1}c_1(g_a)^2c_1(g_b)+2^{\ic-1}c_1(g_a)c_1(g_b)^2-2^{\ic-2}c_1(g_a)^2c_1(g_b)^2.
\end{align*}
From this, it is clear that $\res(\mu)\in \gamma^3\K0(X_s)$. Next we apply the map $q$, 
\begin{align*}
q(\res(\mu))&=2^{\ic -2}y_1^2y_2^2+2^{\ic-1} y_1y_2(y_1+y_2-y_1y_2)
	=2^\ic y_1y_2.
\end{align*}
By the hypothesis that $\ic\leq v_2(n)$, we have $2^\ic y_1y_2\notin\gamma^3_\xi\mathfrak{G}_s$, and hence $\mu\notin\gamma^3\K0(X)$.

Finally, we must check that $\mu$ is torsion. Consider the following elements, which by definition lie in $\res(\gamma^3\K0(X))$:
\begin{align*}
\binom{2^\ia}{2}\binom{2^\ib}{1}c_1(g_a)^2c_1(g_b)&=2^{\ia+\ib-1}c_1(g_a)^2c_1(g_b)\\
\binom{2^\ia}{1}\binom{2^\ib}{2}c_1(g_a)c_1(g_b)^2&=2^{\ia+\ib-1}c_1(g_a)c_1(g_b)^2\\
\binom{2^\ia}{2}\binom{2^\ib}{2}c_1(g_a)^2c_1(g_b)^2&=2^{\ia+\ib-2}c_1(g_a)^2c_1(g_b)^2.
\end{align*}
Thus, 
$$
2^{\ia+\ib-\ic}\mu=2^{\ia+\ib-1}c_1(g_a)^2c_1(g_b)+2^{\ia+\ib-1}c_1(g_a)c_1(g_b)^2-2^{\ia+\ib-2}c_1(g_a)^2c_1(g_b)^2.
$$
\end{proof}

\begin{ex}
Consider again the case that $n=4$. Then, the indices which saitsfy the neccessary hypotheses are $1\leq \ia \leq 2$, $\ib=\ic=2$. So, we have
\begin{align*}
\mu&=c_1(g_a^2)c_1(g_b^2)+c_2(4g_a)^2+c_2(4g_b)-c_2(4g_c)\\
&=2c_1(g_a)^2c_1(g_b)+2c_1(g_a)c_1(g_b)^2-c_1(g_a)^2c_1(g_b)^2.
\end{align*}
Now if $\ia=1$, we have $2^{\ia+\ib-\ic}\mu=2\mu\in\gamma^3\K0(X)$ by the above argument. If instead $\ia=2$, then $2^{\ia+\ib-\ic}\mu=4\mu\in\gamma^3\K0(X)$. In either case $\mu$ is a non-trivial torsion element in $\gamma^{2/3}\K0(X)$. 

An explict example can be constructed for the case that $\ia=1$ and $\ib=\ic=2$. We may define $G=PGO^+(A, \sigma)$ of type $D_4$, by the algebra with involution $(A, \sigma)$, where $(B, \tau)$ and $(C, \gamma)$ are the two components of the Clifford algebra $\mathcal{C}(A, \sigma)$, each endowed with its canonical involution. It follows by the structure theorems \cite{INV} that both are degree 8 algebras with orthogonal involutions. 
The example we present here was contructed in \cite{QSZ}, using the notion of direct sums of agebras with involution introduced by Dejaiffe (cf. \cite{Dj}).

Let $F=K(x,y,z,t)$ be a function field in 4 variables over a field $K$, and consider the quaternion algebras over $F$ defined by
$$
Q_1=(x, zt), \;\;\; Q_2=(y, zt), \;\;\; Q_3=(xy, t), \;\;\; Q_4=(xy, z).
$$
Let $(A, \sigma)$ be the direct sum of $(Q_1, \bar{ })\otimes(Q_3, \bar{ })$ and $(Q_2, \bar{ })\otimes(Q_4, \bar{ })$. Denote by $(B, \tau)$ the component of $\mathcal{C}(A, \sigma)$ Brauer equivalent to $(Q_1,\bar{})\otimes (Q_3,\bar{})\sim (x, t)\otimes (y, z)$, and by $(C, \gamma)$ the component of $\mathcal{C}(A, \sigma)$ Brauer equivalent to $(Q_1, \bar{})\otimes(Q_4, \bar{})\sim (x, z)\otimes (y, t)$. Thus $\ia=v_2(\ind(A))=1$, $\ib=v_2(\ind(B))=2$ and $\ic=v_2(\ind(C))=2$, as required.
\end{ex}
We consider the behaviour of the torsion element $\mu$ over finite field extensions.

\begin{ex}
Consider a finite field extension $l/k$ such that one of the Tits algebras becomes split, that is either $C_+$ or $C_-$ splits over $l$. We may look at the image of our element $\mu$ under the restriction map $\res_{l/k}:\K0(X)\rightarrow \K0(X\times l)$. 

Suppose that $\sigma_1+\sigma_2=\overline{\omega}_{n-1}$, and that $C_+$ becomes split over $l$. Over this field extension we have $\Lambda/T^*=\{0, \sigma=\overline{\omega}_{n}\}$, where $\overline{\omega}_{n-1}=0$ and $\overline{\omega}_1=\overline{\omega}_{n}$. So, $\Lambda/T^*\simeq \ZZ/2\ZZ$, corresponding to \cite[Example 3.6]{Z12}. Let $g=[\LL(\lambda)]$, where $\overline{\lambda}=\sigma$, and let $\istar=v_2(\ind(\beta_\xi(\overline{\lambda})))$, that is, the 2-adic valuation of the index of the Tits algebra of $G$ over $l$. Since indices of Tits algebras cannot increase over field extensions we have $\istar\leq \ic$. With this notation, we have
$$
\res_{l/k}(c_1(g_a))=\res_{l/k}(c_1(g_b))=c_1(g) \;\;\text{ and } \;\; \res_{l/k}(c_1(g_c))=1.
$$
With this, 
$$\res_{l/k}(\mu)=2^\ic c_1(g)^3-2^{\ic-2}c_1(g)^4=2^{\ic-\istar+1}\eta,
$$
where $\eta=2^{\istar-3}(4c_1(g)^3-c_1(g)^4)$, the torsion element defined in \cite{Z12}. Since $\eta$ is a torsion element of order 2, we observe that $\mu$ becomes trivial over this field extension. 
\end{ex}

{\bf Acknowledgments.} 
I am grateful to my PhD supervisor K. Zainoulline for useful discussions on the subject of this paper. This work has been partially supported by his NSERC Discovery grant 385795-2010, by my NSERC Canadian graduate scholarship, and by a graduate scholarship provided by the German Academic Exchange Service.  My special gratitude is due to N. Semenov and the University of Mainz for their hospitality and support.

\bibliographystyle{plain}

\end{document}